\documentclass[reqno,12pt]{amsart}
\usepackage{graphicx, epstopdf}
\usepackage[a4paper,left=1.5cm,right=1.5cm,top=2.5cm,bottom=2cm]{geometry}
\usepackage{amsmath,amssymb,tabularx,setspace,amsthm}
\usepackage{rotate}
\usepackage{lscape}
\usepackage{graphicx}
\usepackage{amsmath}
\usepackage{amsthm}
\usepackage{epsfig}
\usepackage{amssymb}
\usepackage{colortbl}
\usepackage{color}
\newtheorem{theorem}{Theorem}[section]
\newtheorem{lemma}{Lemma}[section]
\numberwithin{equation}{section}

\usepackage{hyperref}

\newcommand{\bc}{\begin{center}}
\newcommand{\ec}{\end{center}}

\newcommand{\be}{\begin{eqnarray}}
\newcommand{\ee}{\end{eqnarray}}

\newcommand{\bea}{\begin{eqnarray*}}
\newcommand{\eea}{\end{eqnarray*}}

\newcommand{\disp}{\displaystyle}

\begin{document}

\title[]{Fractional Brownian Markets with Time-Varying Volatility and High-Frequency Data}
\author[]{A\lowercase{nanya} L\lowercase{ahiri}$^{1,\dag}$ \lowercase{and} R\lowercase{ituparna} S\lowercase{en}$^{2}$\\
$^{1}$C\lowercase{hennai} M\lowercase{athematical} I\lowercase{nstitute}, C\lowercase{hennai}, I\lowercase{ndia}\\
$^{2}$ I\lowercase{ndian} S\lowercase{tatistical} I\lowercase{nstitute},
  C\lowercase{hennai}, I\lowercase{ndia}\\}
\thanks{\dag{\it Corresponding author}, {E-mail: \tt ananya.isi@gmail.com}}

\begin{abstract}
Diffusion processes driven by Fractional Brownian motion (FBM) have often been considered in modeling stock price dynamics in order to capture the long range dependence of stock price observed in reality. Option prices for such models had been obtained by Necula (2002) under constant drift and volatility. We obtain option prices under time varying volatility model. The expression depends on volatility and the Hurst parameter in a complicated manner. We derive a central limit theorem for the quadratic variation as an estimator for volatility for both the cases, constant as well as time varying volatility. That will help us to find estimators of the option prices and to find their asymptotic distributions.


\end{abstract}
\maketitle
\section{\bf Introduction}
It has been proposed to model stock prices as a diffusion driven by  fractional Brownian motion (fBm) in order to capture long range dependence of stock price in reality. See Cont (2005) for evidence of long memory in finance and relation to fractional Brownian motion.

Cheridito (2003) has shown that the solution of the diffusion equation driven by fBm with suitably time lag will lead to an arbitrage-free model. Guasoni (2006) has shown no arbitrage under transaction cost for fBm model. Elliott and Van der Hoek (2003), Biagini et al. (2004) have shown under Wick Ito Skorohod notion of integration one can get arbitrage free market with fBm in some sense. Option prices for such models are obtained by Necula (2002) under constant drift and volatility. One of the aims of this paper is to obtain estimator for some functional of volatility which can be used to price option under time varying volatility model.

For Brownian motion (Bm) setup, the estimation of one of the important functional of volatility appeared in option price formula, called integrated volatility, is performed using sum of frequently sampled squared data. For high frequency data with equal interval this estimator is essentially quadratic variation. FBm is long memory process for Hurst parameter $H\in (\frac{1}{2},1)$. It is well established result that for pure fBm with $H\in (\frac{1}{2},\frac{3}{4})$ quadratic variation is asymptotically normal. Using that result we will show that in the diffusion driven by fBm with $H<\frac{3}{4}$ with constant volatility, quadratic variation is 
asymptotically normal. For similar model and low frequency data with fixed time gap asymptotic normality for volatility estimator was obtained by Xiao et al.(2013). In our paper we consider the estimator for high frequency data with time intervals decreasing to zero and show the asymptotic normality for the estimator. Confidence intervals for volatility can now be translated to confidence intervals for option prices as the expression for option price involve the quantity volatility. For diffusion driven by fBm with $H\in (\frac{1}{2},\frac{3}{4})$ and time varying bounded volatility case also we will show the asymptotic normality of the  estimator from high frequency data. 

The objective of this paper is two fold. Firstly for the diffusion driven by fBm with time varying volatility we will find the option price in terms of some functional of volatility and the Hurst parameter. Secondly we will show the asymptotic normality property for the estimator for such parametric function. The estimator requires the prior knowledge of Hurst parameter. Once the estimate of functional of volatility is found one can apply the estimate to get the option price. 

The rest of the paper is organized as follows. In section \ref{sec2} we describe the diffusion model for the stock price. In section \ref{sec3} we introduce the proper notion of integration required to obtain an arbitrage-free solution of the diffusion equation.  In section \ref{sec4} we present the option pricing results. The central limit theorems for the proposed estimator are obtained in section \ref{sec5}.  We conclude and summarize the current and future research directions in section \ref{sec6}.

\section{\bf Model}\label{sec2}

The introduction of the fractional Black-Scholes model, where the
Bm in the classical Black-Scholes model is replaced by
a fBm, have been motivated by empirical studies (see for example
Mandelbrot (1997), Shiryaev (1999)).
The risk free asset equation is 
\begin{equation}dP_t=r_tP_tdt,\quad P_0=1, \quad 0\le t\le T\label{eqn01}
\end{equation}
The risky asset equation is 
\begin{equation}dS_t=\mu_tS_tdt+\sigma_tS_tdB^H_t,\quad S_0=S>0, \quad 0\le t\le T\label{eqn02}
\end{equation}
where $B^H_t$ is FBM with initial condition $S_0 = S>0.$ Here $H$ is Hurst parameter, for $0 < H < 1$. $\mu_t$ is real valued deterministic function of time $t$, called drift and $\sigma^2_t$ is positive real valued deterministic function of $t$, called volatility. $B^H_t$ is a continuous and centered
Gaussian process starting at 0 with covariance and variance
functions as follows: $\forall H \in (0, 1), s,t >0$
\begin{equation}\mathrm{E}(B^H_t B^H_s) =\frac{1}{2}(t^{2H} + s^{2H} - \mid t - s\mid^{2H})
\end{equation}
\begin{equation} E(B^H_t)^2 = t^{2H},  
\end{equation}
$B^H_t$ has homogeneous increments, i.e., $B^H_{t+s} -B^H_s$ has same law as $B^H_t$ for all $s, t > 0.$
Increments are dependent and correlation between the increments $B^H_{t+h} - B^H_t$ and $B^H_{s+h} - B^H_s$ with $s + h \le t$ and $t - s = nh$ is as follows:
\begin{equation}\rho_n^H =\frac{1}{2}h^{2H}[(n + 1)^{2H} + (n - 1)^{2H} - 2n^{2H}]
\end{equation}

Firstly we will find the European call option price for this model. Secondly we try to provide an estimator for option price. In this process we see that it is enough to study the quadratic variation of this process for given high frequency data $0=t_0<t_1<\cdots<t_N=1$ with $S_{t_j}, \ j=0,\cdots, N$ and $t_{j+1}-t_{j}=\disp\frac{1}{N}\ \forall \ j=0,\cdots, N-1$. We will propose suitable estimator with this data and see where it converges and how that is useful for estimating option price. We observe that the analysis is based on high frequency data, as sample size increases the time difference between two consecutive data point decreases. We also note that through out our analysis we know $H$, we do not estimate $H$ from data.
\section{\bf Regarding the solution of the SDE, Wick Ito Skorohod Integral, $H\in(0,1)$}\label{sec3}

In order to find the solution of the diffusion equation (\ref{eqn02}), we need to note that the whole
analysis depends on how we interpret the term $dB^H_t$. For $H \ne 1/2 , B^H_t$ is not a semimartingale.

There are different notions of integration with respect to $B^H_t$,  out of which
we choose Wick Ito Skorohod (WIS)
integral notion to solve equation (\ref{eqn02}), for $H\in (0,1)$, due to
financial reason outlined in, for example, see Elliot and Van der Hoek (2003) and Biagini $et.al.$ (2004),(2008).

Rogers (1997) explains why other common notions of integral are inappropriate.  

Following the WIS notion of integration the solution of the stochastic differential equation (\ref{eqn02}) is
\begin{equation}S_t = S_0 \exp\left(\int^t_0\sigma_sdB^H_s + \int^t_0\mu_sds -\frac{1}{2} \int_\mathbb{R}(M(\sigma_s\chi_{[0,t]}))^2ds\right) \label{eqn4}
\end{equation}
where $M$ is an operator acting on $s$ and depends on $H$;
$\chi_{[0,t]}$ is the indicator function. See Biagini et al. (2004, 2008). We will discuss about the operator $M$ in next subsection and meaning of $\int^t_0\sigma_sdB^H_s $ in following subsection.
\subsection{\bf The integral operator $M$ }

Let us elaborate about the operator $M$ which will be needed in future sections, See Biagini et al. (2004). Let $\mathcal{S}(\mathbb{R})$ denote the Schwartz space of smooth rapidly decreasing functions on $\mathbb{R}$.
 $M$ is defined on $\mathcal{S}(\mathbb{R})$ to $L^2(\mathbb{R})$ as follows:
\begin{equation}
Mf(x)=-\frac{d}{dx}C_H\int_{\mathbb{R}}(t-x)|t-x|^{H-\frac{3}{2}}f(t)dt
\end{equation}
$C_H$ is constant.
\begin{equation}
{Mf}(y)=|y|^{\frac{1}{2}-H} \hat{f}(y), y \in \mathbb{R}
\end{equation}
with Fourier transform $\hat{f}$ defined as
\begin{equation}
\hat{f}(y)=\int_{\mathbb{R}}e^{-ixy}f(x)dx.
\end{equation}
It turns out that $\disp C_H=\frac{[\Gamma(2H+1)\sin(\pi H)]^\frac{1}{2}}{[2 \Gamma(H-\frac{1}{2}) \cos (\frac{1}{2}\pi (H-\frac{1}{2}))]}$, $\Gamma(.)$ is gamma function and explicit expression for $M$ is as follows:
\begin{equation}
\mbox{for} \ H\in (0,\frac{1}{2}),\  Mf(x)=C_H\int_{\mathbb{R}}\frac{f(x-t)-f(x)}{|t|^{\frac{3}{2}-H}}dt
\end{equation}
\begin{equation}
\mbox{for} \ H=\frac{1}{2},\  Mf(x)=f(x)
\end{equation}
\begin{equation}
\mbox{for} \ H\in (\frac{1}{2},1), \  Mf(x)=C_H\int_{\mathbb{R}}\frac{f(t)}{|t-x|^{\frac{3}{2}-H}}dt
\end{equation}
$M$ extends $\mathcal{S}(\mathbb{R})$ to $L^2_H$ where
\begin{eqnarray*}
L^2_H(\mathbb{R})&=& \{f:\mathbb{R}\rightarrow \mathbb{R}(\mbox{deterministic});|y|^{\frac{1}{2}-H} \hat{f}(y)\in L^2(\mathbb{R})\}\\
&=& \{f:\mathbb{R}\rightarrow \mathbb{R};Mf(x)\in L^2(\mathbb{R})\}\\
&=& \{f:\mathbb{R}\rightarrow \mathbb{R}; \|f\|_{L^2_H}<\infty \}, \mbox{where}  \|f\|_{L^2_H}=\|Mf\|_{L^2}\}
\end{eqnarray*}
We also have for $f \in L^2_H(\mathbb{R})$ 
\begin{equation}
\langle f,g\rangle_{L^2_H}=\langle Mf,Mg\rangle_{L^2}
\end{equation}
and for $f,g \in L^2(\mathbb{R})\cap L^2_H(\mathbb{R})$
\begin{equation}
\langle f, Mg\rangle_{L^2}=\langle \hat f,\widehat {Mg}\rangle_{L^2}=\langle \widehat {Mf}, \hat g\rangle_{L^2}=\langle Mf, g\rangle_{L^2}
\end{equation}
\subsection{\bf Wiener Integral with respect to FBM, $H\in(0,1)$}\label{secwi}

Let $f\in L^2_H(\mathbb{R})$, deterministic. Then $Mf \in L^2(\mathbb{R})$. 
The Wiener integral with respect to fractional Brownian motion is defined as
\begin{equation}
\int_{\mathbb{R}}f(s)dB_s^H =\int_{\mathbb{R}}(Mf)(s)dB_s
\end{equation}

For detail see Appendix.
\section{\bf On the way to calculate European call option price\label{sec4}}
In this section we follow similar line of argument in that of Elliot and Van der Hoek (2003). 

\subsection{Risk-neutral measure}
Let us write equation (\ref{eqn02}) with the notion if Wick product. Now following Elliott and Van der Hoek (2003) we rewrite equation (\ref{eqn02}) as follows:
\begin{equation}
dS_t=S_t\diamond [\mu_t +\sigma_t W^H_t]dt \label{eqn8}
\end{equation} 
where $\diamond$ is the Wick product for two processes and $\disp W^H_t=\frac{dB^H_t}{dt}$. For meaning of Wick product and $W^H_t$ see Appendix (\ref{WP}). From theorem of ``Wick Ito integral" Biagini {\it et al.} (2008) or (Appendix (\ref{WP})) we note that $S_tdB^H_t=S_t\diamond W^H_t dt$. 
We denote trading strategy or portfolio as $\theta(t,\omega)=\theta(t)=(u(t),v(t))=(u_t,v_t)$ where $u(t)$ and $v(t)$ are the number of units of bond and stock respectively in the portfolio at time $t$ and the processes are adaptive. The value process is defined as
\begin{equation}
z_t^{\theta}=u_tP_t+v_t S_t \label{eqn1}
\end{equation}
{\bf Definition}: The concept analogous to self-financing in the fractional Brownian setting is Wick-financing. Elliott and Van der Hoek mention it as self financing but Wick financing strategy is not usual buy and hold strategy. A portfolio is {\bf Wick-financing} if
\begin{equation}
dz^{\theta}_t=u_tdP_t+v_t dS_t= u_tdP_t+v_t S_t \diamond [\mu_t +\sigma_t W^H_t]dt\label{eqn2}
\end{equation}
\begin{eqnarray*}dz_t^{\theta}&=&u_tdP_t+v_t S_t\diamond(\mu_t+\sigma_t W^H_t)dt \\
&=&u_tr_tP_td_t+v_t S_t\diamond \mu_t dt+\sigma_tv_t S_t \diamond  W^H_t dt\\
&=&(z_t-v_t S_t)r_td_t+\mu_tv_t S_tdt+\sigma_tv_tS_t \diamond  W^H_t dt\\
&=&z_tr_tdt+\sigma_tv_t S_t\diamond \left(\frac{\mu_t-r_t}{\sigma_t}+ W^H_t\right)dt\\
\end{eqnarray*}
From the Girsanov theorem in Elliott and Van der Hoek (2003) the translated process $\hat{B}^H_t$ as
\begin{equation}
\hat{B}^H_t=\int_0^t\frac{\mu_s-r_s}{\sigma_s}d_s+B^H_t \label{eqn3}
\end{equation} is fBm with respect to new measure $\hat P$ defined on $\mathcal{F}$ by $ \disp  \frac{d\hat P}{d P}=\exp(\langle \omega, \phi \rangle -\frac{1}{2}\|\phi\|^2_{L^2})=\exp(\int_{\mathbb{R}}\phi(s)dB_s-\frac{1}{2}\|\phi\|^2_{L^2})$
where $\disp \phi (s) =M^{-1} \left(\frac{r(s)-\mu(s)}{\sigma(s)}\right)I_{[0,t]}(s) $. For notation $\exp(\langle \omega, \phi \rangle -\frac{1}{2}\|\phi\|^2_{L^2})$ see Appendix (\ref{WI}). We note that $\phi(s)$ has to be in $L^2{\mathbb{R}}$.

Now we can rewrite \ref{eqn2} as
\begin{equation}
dz_t=r_tz_tdt+\sigma_tv_t S_t \diamond \hat{W}_t^H dt
\end{equation} where, $\disp \hat{W}^H_t=\frac{d\hat{B}^H_t}{dt}$. 
Multiplying both sides with $\exp(\disp\tilde{r}_t)$ with $\disp\tilde{r}_t=\int_0^t r_s ds$ and integrating, we get
\begin{equation} e^{-\tilde{r}_t}z_t-z_0=\int_0^t e^{-\tilde{r}_s}\sigma_s v_s S_s \diamond \hat{W}_s^H ds\label{eqn6}
\end{equation}
and
\begin{equation*}
\hat E \left[e^{-\tilde{r}_T}z_T\right]=z_0
\end{equation*} where $\hat E$ is expectation under measure $\hat P$.
Thus there exists a risk-neutral measure.

We note that under risk neutral measure $\hat{P}$ we have
\begin{equation}
dS_t=r_t S_t dt+\sigma_t S_t d\hat{B}^H_t
\end{equation}
which will be useful for calculating option price.
\subsection{Complete Market}
Let $\mathcal{F}_t=\sigma\{B^H_s,\ 0\leq s \leq t\}$ be the filtration.  The market is complete if $\forall  \ \mathcal{F}_T $ measurable bounded random variable $F$, $\exists \ z\in \mathbb{R}$ and portfolio $(u_t,v_t)$ such that $F=z_T$ almost surely $\hat P$, where $z_T$ is given by \ref{eqn1}. We now proceed to verify this. By fractional Clark-Ocone theorem in Elliott and Van der Hoek (2003) applied to $F$, we have,
\begin{equation}
e^{-\tilde{r}_T}F=
\hat E\left[e^{-\tilde{r}_T}F\right]
+\int_0^T \tilde{E}_{\hat{P}}\left[\hat{D}_t(e^{-\tilde{r}_T}F)\mid \mathcal{F}_t \right] \diamond \hat{W}_t^H dt\label{eqn7}
\end{equation}
Here $\tilde{E}_{\hat{P}}$ denotes the quasi-conditional expectation and $\hat{D}_t$ is the fractional Hida Malliavin derivative with respect to $\hat{B}^H_t$. For detail see  Elliot and van der Hoek (2003)and Biagini {\it et. al.} (2008).
We take $z=\hat E\left[e^{-\tilde{r}_T}F\right]$. Now comparing equations (\ref{eqn6}) and (\ref{eqn7}) we get
\begin{equation*}
\tilde{E}_{\hat{P}}\left[\hat{D}_t(F)\mid \mathcal{F}_t\right]=e^{\tilde{r}_T-\tilde{r}_t}\sigma_tv_t S_t
\end{equation*}
This is the condition for completeness of the market. Here we note that there is criticism about this notion of completeness with Wick financing instead of self financing, see Bjrk and Hult (2005). Fractional Black Scholes market has weak arbitrage but no strong arbitrage, see Biagini et al.(2008).

In the context of quasi conditional expectation we require following lemma which will be useful for calculating option price in next section.
\begin{lemma}\label{qce}
a) If $\disp g_t=\exp\big(\int_0^t\sigma_s d\hat{B}^H_s- \frac{1}{2} \int_\mathbb{R}(M(\sigma_s\chi_{[0,t]}))^2ds\big)$ then for $T>t$, $\tilde{E}_{\hat{P}}[g_T| \mathcal{F}_t]=g_t$.

b) If $F\in L^{1,2}(\hat{P})$ (similar to Definition A4 of Elliott and Van der Hoek (2003)), and $\disp G_t= \int_0^t F_td\hat{B}^H_t$, 
then for $T>t$, $\tilde{E}_{\hat{P}}[G_T| \mathcal{F}_t]=G_t$.
\end{lemma}
\begin{proof}
Proof can be done by direct calculation.
\end{proof}
\subsection{Price of European Call Option}
We next will find the European call option price $\tilde{E}_{\hat{P}}[(S_T - K)_+|\mathcal{F}_t]$ for this model. When $\mu_t = \mu$ and $\sigma_t = \sigma$,  Necula (2008) obtains the price $C$ at every $t \in (0,T)$ of an European call option with strike price $K$ and
maturity $T$ as \begin{eqnarray}
&C(t, S_t)= S_t \Phi(d_1) -Ke^{-r(T-t)}\Phi(d_2)&\\
\mathrm{where}\quad&
d_1=\frac{\log(\frac{S_t}{K})+r(T-t)+\frac{\sigma^2}{2}(T^{2H}-t^{2H})}{\sigma\sqrt{T^{2H}-t^{2H}}}&\quad \mathrm{
and}\\&
d_2=\frac{\log(\frac{S_t}{K})+r(T-t)-\frac{\sigma^2}{2}(T^{2H}-t^{2H})}{\sigma\sqrt{T^{2H}-t^{2H}}}
&\end{eqnarray}
and $\Phi(×)$ is the cumulative probability of the standard normal distribution.

The confidence intervals for $\sigma^2$ obtained in section \ref{sec5} can be translated to prediction intervals for $C$ as in Mykland (2000) or Avellaneda et al (1995).

Next for time varying $\mu_t$ and $\sigma_t$ let us calculate option price. We need the following lemma.

\begin{lemma}
The price at every $t\in[0,T]$ of bounded $\mathcal{F}_T$ measurable function $F\in L^2(\hat{P})$ is given by $F(t)=\exp(-\tilde{r}_T+\tilde{r}_t) \tilde{E}_{\hat{P}}[F|\mathcal{F}_t] $. 
\end{lemma}
\begin{proof}
Proof can be followed in similar line as in Theorem 4.1 from Necula (2008) and using part b) of lemma \ref{qce}.
\end{proof}
For European call option price $F$ will be $F(\omega)=(S(T,\omega)-K)^+$ where $K$ is the strike price.

Now following similar line of approach from Theorem 3.1 of Necula (2008) and using part a) of lemma \ref{qce} we get
\begin{equation}\label{cet}
\tilde{E}_{\hat{P}}\left(\exp\left(\int_0^T\sigma_sd\hat{B}^H_s\right)|\mathcal{F}_t\right)=
\exp\left(\int_0^t\sigma_sd\hat{B}^H_s-\frac{1}{2}\int_\mathbb{R}(M(\sigma_s\chi_{[{0,t}]}))^2ds+\frac{1}{2}\int_\mathbb{R}(M(\sigma_s\chi_{[{0,T}]}))^2ds\right)
\end{equation}
Equation (\ref{cet}) will be used for proving next theorem.
\begin{theorem}
The price at every $T\in[0,T]$ of an European call option with strike price $K$ and maturity $T$ is given by $S_t\Phi(d_1)-K \exp(-\tilde{r}_T+\tilde{r}_t)\Phi(d_2)$ where $$d_1=\disp\frac{\ln (S_t/K)+\tilde{r}_T-\tilde{r}_t+\int_\mathbb{R}(M(\sigma_s\chi_{[{0,T}]}))^2ds-\int_\mathbb{R}(M(\sigma_s\chi_{[{0,t}]}))^2ds}{\sqrt{\int_\mathbb{R}(M(\sigma_s\chi_{[{0,T}]}))^2ds-\int_\mathbb{R}(M(\sigma_s\chi_{[{0,t}]}))^2ds}}$$ and $$d_2=\disp\frac{\ln (S_t/K)+\tilde{r}_T-\tilde{r}_t-\int_\mathbb{R}(M(\sigma_s\chi_{[{0,T}]}))^2ds+\int_\mathbb{R}(M(\sigma_s\chi_{[{0,t}]}))^2ds}{\sqrt{\int_\mathbb{R}(M(\sigma_s\chi_{[{0,T}]}))^2ds-\int_\mathbb{R}(M(\sigma_s\chi_{[{0,t}]}))^2ds}} $$.
\end{theorem}
\begin{proof}
Proof is in similar line as that of given in Necula (2008).
\end{proof}

\section{\bf Estimation of volatility from discrete observations}\label{sec5}
Assume that the process is observed at discrete-time instants $0=t_0<t_1< t_2< \cdots < t_N=T$.
Thus the observation vector is $S = (S_{t_1} , S_{t_2} , \cdots, S_{t_N} )$. We note that this is high frequency data. In particular, we assume $t_k = kh, k = 1, 2,\cdots ,N$ for a step size, $\disp h=\frac{1}{N} > 0$.
In section \ref{sec:const} we present the results when $\sigma_t$ is constant and in section \ref{sec:tv} when $\sigma_t$ is time varying. In section \ref{sec:Sim1} and \ref{sec:sim2} we present simulation studies.
\subsection{Constant $\sigma$}\label{sec:const}
Let us start with the estimator of $\sigma^2$ as
\begin{equation}\label{eqn5}\hat{\sigma^2} =\frac{1}{Nh^{2H}}\sum^{N-1}_{k=0}\left(\log(S_{t_{k+1}})-\log(S_{t_k})\right)^2\end{equation}

We shall prove a central limit theorem for $\hat{\sigma^2}$. The main component of the proof is central limit theorem for quadratic variations of fractional Brownian motion. We also need to bound the additional terms that comes from the geometric nature of our process. Some of these arguments are similar to those of Nourdin (2008, 2009). The main theorem for this section is given below.
\begin{theorem}\label{thm1} Assume that the stock price follows the diffusion model specified by equation (\ref{eqn3}) with no drift and constant volatility $\sigma$. Also assume that $H\in (0,3/4)$. $N\rightarrow\infty$ with the observation interval $Nh=T$ remaining constant. Without loss of generality we can assume $T=1$.
Then \begin{equation}\sqrt{N}(\hat{\sigma^2}-\sigma^2) \Longrightarrow \mathcal{N}(0,\sigma^2_{H,2})\end{equation} where $\sigma^2_{H,2}$ is a constant that can be computed explicitly, $$\disp \sigma^2_{H,2}=2\sigma^4\lim_{N\rightarrow\infty}(1+2(1-\frac{1}{N})(2^{2H}-1)^2+\sum_{k=2}^N(1-\frac{k}{N})[(k+1)^{2H}+(k-1)^{2H}-2k^{2H}]^2).$$
\begin{proof}  Under the condition of $\mu=0$ and $\sigma_t=\sigma$, the solution (\ref{eqn4}) of the stochastic differential equation (\ref{eqn1}) simplifies to \begin{equation}S_t = S_0 \exp\left(\sigma dB^H_t  -\frac{1}{2} \sigma^2 t^{2H}\right).\end{equation}  Putting this solution in the definition of $\hat{\sigma^2}$ in equation (\ref{eqn5}), we get,
\begin{eqnarray}\label{eqn10}\hat{\sigma^2}& =&\frac{1}{Nh^{2H}}\sum^{N-1}_{k=0}\left[\sigma\left(B_{(k+1)h}^H- B_{kh}^H \right) -\frac{\sigma^2}{2}\left(\{(k+1)h\}^{2H}-\{kh\}^{2H}\right)\right]^2\nonumber\\
&=&\frac{1}{Nh^{2H}}\sigma^2\left[\begin{array}{l}\sum^{N-1}_{k=0}\left(B_{(k+1)h}^H- B_{kh}^H \right)^2 \\-\sigma \sum^{N-1}_{k=0}\left\{\left(B_{(k+1)h}^H- B_{kh}^H \right)\left(\{(k+1)h\}^{2H}-\{kh\}^{2H}\right) \right\}\\ +\frac{\sigma^2}{4}\sum^{N-1}_{k=0}\left(\{(k+1)h\}^{2H}-\{kh\}^{2H}\right)^2\end{array}\right]
\end{eqnarray}
It is already known, for example putting $\kappa=2$ in equation (1.5) of Nourdin(2008), that if $H\in (0,3/4)$, 
\begin{equation}\label{eqn11}X_N:=\frac{1}{\sqrt{N}}\sum^{N-1}_{k=0}\left[h^{-2H}\left(B_{(k+1)h}^H- B_{kh}^H \right)^2 -1\right]\Longrightarrow \mathcal{N}(0,\sigma^2_{H,2})
\end{equation}
Combining (\ref{eqn10}) and (\ref{eqn11}), we have
\begin{eqnarray}\sqrt{N}(\hat{\sigma^2}-\sigma^2)&=&X_N-\frac{\sigma^3}{\sqrt{N}h^{2H}} \sum^{N-1}_{k=0}\left\{\left(B_{(k+1)h}^H- B_{kh}^H \right)\left(\{(k+1)h\}^{2H}-\{kh\}^{2H}\right) \right\}\nonumber\\&& +\frac{\sigma^4}{4\sqrt{N}h^{2H}}\sum^{N-1}_{k=0}\left(\{(k+1)h\}^{2H}-\{kh\}^{2H}\right)^2.
\end{eqnarray}
It is shown in lemma \ref{lem1} that the second term converges to zero in $L^2$ as $N\rightarrow\infty$. In lemma \ref{lem2} it is shown that the third term converges to zero. The theorem now follows by applying Chebyshev's inequality to get convergence in $L^2$ implies convergence in probability and Slutsky's theorem to get final asymptotic normality.
\end{proof}
\end{theorem}

\begin{lemma}\label{lem1} Under the assumptions of theorem \ref{thm1}, \begin{equation}
\frac{\sigma^3}{\sqrt{N}h^{2H}} \sum^{N-1}_{k=0}\left\{\left(B_{(k+1)h}^H- B_{kh}^H \right)\left(\{(k+1)h\}^{2H}-\{kh\}^{2H}\right) \right\}\stackrel{L^2}{\longrightarrow} 0 \label{cl2}
\end{equation}\end{lemma}
\begin{proof}
As $\disp h=\frac{1}{N}$ l.h.s of (\ref{cl2}) is
\begin{equation}
U_1=\sigma^3 \disp N^{2H-\frac{1}{2}}\sum^{N-1}_{k=0}(B_{\frac{k+1}{N}}^H- B_{\frac{k}{N}}^H )((\frac{k+1}{N})^{2H}-(\frac{k}{N})^{2H}) \label{cl22}
\end{equation}
Now we have $E(B^H_{\frac{k+1}{N}}-B^H_{\frac{k}{N}})^2=\disp\frac{1}{N^{2H}}$,
$$E(B^H_{\frac{k+1}{N}}-B^H_{\frac{k}{N}})(B^H_{\frac{k+2}{N}}-B^H_{\frac{k+1}{N}})=\frac{2^{2H}-1}{N^{2H}}\sim \frac{1}{N^{2H}}$$
$$E(B^H_{\frac{k+1}{N}}-B^H_{\frac{k}{N}})(B^H_{\frac{j+1}{N}}-B^H_{\frac{j}{N}})=\frac{1}{2}\Big[|\frac{k-j+1}{N}|^{2H}+|\frac{k-j-1}{N}|^{2H}-2|\frac{k-j}{N}|^{2H} \Big] \mbox{for}\  |k-j|>1.$$ Then second moment of (\ref{cl22}) becomes
\begin{eqnarray*}
&&\sigma^6 \disp N^{4H-1}\sum^{N-1}_{k=0}\sum^{N-1}_{j=0}((\frac{k+1}{N})^{2H}-(\frac{k}{N})^{2H})\big((\frac{j+1}{N})^{2H}-(\frac{j}{N})^{2H}\big)E(B_{\frac{k+1}{N}}^H- B_{\frac{k}{N}}^H )(B_{\frac{j+1}{N}}^H- B_{\frac{j}{N}}^H )\\
&=&\frac{\sigma^6}{N} \mathop{\disp\sum_{k=0}^{N-1}\sum_{j=0}^{N-1}}_{k=j, |k-j|=1}(2H(k+1)^{2H-1})(2H(j+1)^{2H-1}) \times [\frac{1}{N^{2H}}]\\
&&\qquad +\frac{\sigma^6}{N} \mathop{\disp\sum_{k=0}^{N-1}\sum_{j=0}^{N-1}}_{|k-j|>1} (2H(k+1)^{2H-1}) (2H(j+1)^{2H-1})\\
&& \qquad \qquad \qquad \qquad \times\frac{1}{2}\Big[|\frac{k-j+1}{N}|^{2H}+|\frac{k-j-1}{N}|^{2H}-2|\frac{k-j}{N}|^{2H} \Big]
\end{eqnarray*}
\begin{eqnarray*}
&\le & \frac{\sigma^6}{N^{2H+1}} 4H^2[N^{4H-1}+2N^{4H-1}]+\frac{\sigma^6}{N}\mathop{\disp\sum_{k=0}^{N-1}\sum_{j=0}^{N-1}}_{|k-j|>1}(2H(k+1)^{2H-1}) (2H(j+1)^{2H-1})\\
&& \qquad \qquad \qquad \qquad \qquad \qquad \qquad \qquad \times \Big[2\frac{1}{N^2}2H(2H-1)|\frac{k-j}{N}|^{2H-2}\Big]\\
&\le & 12\sigma^6 H^2 [N^{2H-2}]+\frac{16\sigma^6 H^3(2H-1)}{N^{2H+1}}\mathop{\disp\sum_{k=0}^{N-1}\sum_{j=0}^{N-1}}_{|k-j|>1}(k+1)^{2H-1}(j+1)^{2H-1}|k-j|^{2H-2}\\
&=&12\sigma^6 H^2 [N^{2H-2}]+\frac{16\sigma^6 H^3(2H-1)}{N^{2H+1}} T_1
\end{eqnarray*}
Now, \begin{eqnarray*}
T_1&=& \sum_{k=1}^{N-1}k^{2H-2}\sum_{j=1}^{N-k}(2H j^{2H-1})(2H (k+j)^{2H-1})\\
&\le & \sum_{k=1}^{N-1}k^{2H-2}\sum_{j=1}^{N-k}(2H(k+j)^{2H-1})^2\\
&=& \sum_{k=1}^{N-1}k^{2H-2}\sum_{j=k+1}^{N}(2H j^{2H-1})^2\\
&\le & \sum_{k=1}^{N-1}k^{2H-2}\sum_{j=1}^{N}(2H j^{2H-1})^2\\
&\le & \sum_{k=1}^{N-1}k^{2H-2} (4H^2 N^{4H-1})\\
&\le & N^{2H-1}4H^2 N^{4H-1}\\
&=& 4H^2 N^{6H-2}
\end{eqnarray*}

So, we get $E(U_1^2)<12\sigma^6H^2N^{2H-2}+16\sigma^6 H^3(2H-1)N^{4H-3}\rightarrow 0$ as $N\rightarrow\infty$ if $H<\frac{3}{4}$. Hence the result.
\end{proof}

\begin{lemma}\label{lem2} Under the assumptions of theorem \ref{thm1},
\begin{equation}\frac{\sigma^4}{4\sqrt{N}h^{2H}}\sum^{N-1}_{k=0}\left(\{(k+1)h\}^{2H}-\{kh\}^{2H}\right)^2\rightarrow 0 \label{cp}
\end{equation}\end{lemma}
\begin{proof}
Again putting $\disp h=\frac{1}{N}$ l.h.s of (\ref{cp}) is $\disp U_2=\frac{\sigma^4}{4} N^{2H-\frac{1}{2}}\sum^{N-1}_{k=0}(|\frac{k+1}{N}|^{2H}-|\frac{k}{N}|^{2H})^2$
\begin{eqnarray*}
U_2&\le & \disp \frac{\sigma^4}{4} N^{-2H-\frac{1}{2}}\sum^{N-1}_{k=0}4H^2[(k+1)^{2H-1}]^2\le \disp \frac{\sigma^4}{4} N^{-2H-\frac{1}{2}}4H^2 N^{4H-1}=\sigma^4H^2N^{2H-\frac{3}{2}}
\end{eqnarray*}
So, $U_2\rightarrow 0$ as $n\rightarrow\infty$ if $2H-\frac{3}{2}<0$ that is $H<\frac{3}{4}$.
Hence the result.
\end{proof}
We note that $\disp \sigma^2_{H,2}=2\sigma^4\lim_{N\rightarrow\infty}(1+2(1-\frac{1}{N})(2^{2H}-1)^2+\sum_{k=2}^N(1-\frac{k}{N})[(k+1)^{2H}+(k-1)^{2H}-2k^{2H}]^2)$.
\subsection{Simulation Studies for fixed $\sigma$}\label{sec:Sim1}
In this section we present the simulation result for FBM driven model and it's estimator. We use somebm packages from R to simulate fractional Brownian motion. We keep drift parameter $\mu=0$. We generate each sample paths with $N=1000$ points and replicate with replication number, say $r=$, 200 times to find the mean, variance and mean squared error. We repeat the simulation for different values of $\sigma^2$ and for $\disp H\in (\frac{1}{2},\frac{3}{4})$. Simulation shows that estimators are excellent when $\disp H>\frac{1}{2}$ and $\disp H<\frac{3}{4}$ for high frequency data with 1000 values in the time interval 0 to 1. 

\begin{table}
\caption{The MEAN, VAR, MSE, ASYV of the estimators when $\sigma^2$= 0.4}
\begin{center}
\begin{tabular}{|c|c|c|c|} \hline
  \multicolumn{4}{|c|}{$\sigma^2$=0.4} \\  \hline
     H & 0.55   & 0.65   & 0.74 \\
   MEAN &  0.399936  & 0.3986239   & 0.4016842   \\
   VAR & 0.0002994642	&   0.00044069611 &  0.0007573279  	  \\ 
   MSE & 0.0002994683	&   0.0004088546 &   0.0007601643	\\ \hline
\end{tabular}
\end{center}

\caption{The MEAN, VAR, MSE, ASYV of the estimators when $\sigma^2$= 1.6}
\begin{center}
\begin{tabular}{|c|c|c|c|} \hline
  \multicolumn{4}{|c|}{$\sigma^2$=1.6} \\  \hline
     H & 0.55   & 0.65   & 0.74 \\
   MEAN & 1.601847   &  1.605567  &  1.607953  \\
   VAR & 0.00456971	&  0.008053313  &  0.01498881 	  \\ 
   MSE & 0.004573123	&  0.008084309    & 0.01505207	  \\ \hline
\end{tabular}
\end{center}

\caption{The MEAN, VAR, MSE, ASYV of the estimators when $\sigma^2$= 6.4}
\begin{center}
\begin{tabular}{|c|c|c|c|} \hline
  \multicolumn{4}{|c|}{$\sigma^2$=6.4} \\  \hline
     H & 0.55   & 0.65   & 0.74  \\
   MEAN &  6.401082  & 6.450209   &   6.720847 \\
   VAR & 0.08589091	&   0.1253994 &   0.4416412 	  \\ 
   MSE & 0.08589208	& 0.1279204   &  0.5445839 	  \\ \hline
\end{tabular}
\end{center}
\end{table}
\subsection{Time varying $\sigma_t$}\label{sec:tv}
In this section we will prove the result for long memory process only, i.e. $\disp H>\frac{1}{2}$. Our parameter of interest would be $\disp \theta=\int_{\mathbb{R}}(M(\sigma_s\chi_{[0,1]}))^2 ds$.

To get the properties of time varying volatility estimator we need some Mathematical foundation. Readers are referred to Appendix for background for time varying volatility estimator before starting of this section.
We denote
\begin{equation*}\disp \eta_k:=\int_\mathbb{R}\sigma_s\chi_{[\frac{k}{N},\frac{k+1}{N}]}dB^H_s=\int_{\mathbb{R}}f_k(s)dB_s^H=I_1(f_k)
\end{equation*} 
where $f_k(s)=\sigma_s\chi_{[\frac{k}{N},\frac{k+1}{N}]}(s)$ and $I_1$ is Wiener integral with respect to fBm $B^H_t$ so $\eta_k$ is same as the Wiener integral discussed before.  
Define
\begin{equation}
X_N=\sqrt{N}\left(\sum^{N-1}_{k=0}h^{-2H+1}\eta_k^2-\tilde{\theta}_N\right) \label{clt}
\end{equation}
where $\disp \tilde{\theta}_N=\sum^{N-1}_{k=0}h^{-2H+1}E(\eta_k^2)$.

\begin{theorem}\label{thm2} Assume that the stock price follows the diffusion model specified by equation (\ref{eqn3}) with no drift $(\mu=0)$ and time varying volatility $\sigma_t$. Also assume that $H\in (1/2,3/4)$. $N\rightarrow\infty$ with the observation interval $Nh=T$ remaining constant. Without loss of generality we can assume $T=1$.
Then, \begin{equation}\label{eqn13}X_N\Longrightarrow \mathcal{N}(0,\sigma^2_{H,2,*}) 
\end{equation} where $\sigma^2_{H,2,*}$ can be computed explicitly given the form of $\sigma(t)$ with the formula $$\sigma^2_{H,2,*}=\disp\lim_{N\rightarrow \infty}N^{4H-1}2\sum^{N-1}_{k=0}\sum^{N-1}_{k'=0} (\int_{\mathbb{R}}Mf_k(s)Mf_{k'}(s)ds)^2$$. 
\end{theorem}
\begin{proof}
r.h.s. of \ref{clt} can be rewritten as $\disp X_N=\sqrt{N}\left(N^{2H-1}\sum^{N-1}_{k=0}\eta_k^2- \tilde{\theta}_N\right)$.

Let us introduce some notations.

Let us denote $E(\eta_k^2)=\theta_k$. We get $\disp \theta_k=\int_\mathbb{R}(M(\sigma_{s}\chi_{[\frac{k}{N},\frac{k+1}{N}]}))^2ds$ as expectation of Wiener integral $\disp \int_\mathbb{R}\sigma_s\chi_{[\frac{k}{N},\frac{k+1}{N}]}dB^H_s$.
Now $$E\disp (N^{2H-1}\sum^{N-1}_{k=0}\eta_k^2)=N^{2H-1}\sum^{N-1}_{k=0}\theta_k=N^{2H-1}\sum^{N-1}_{k=0}\int_\mathbb{R}(M(\sigma_{s}\chi_{[\frac{k}{N},\frac{k+1}{N}]}))^2ds$$
$$\disp N^{2H-1}\sum^{N-1}_{k=0}\int_\mathbb{R}(M(\sigma_{s}\chi_{[\frac{k}{N},\frac{k+1}{N}]}))^2ds=\tilde{\theta}_N $$%

If $H>1/2$ and $\sigma(s)<\Sigma \ \forall \ s$ then
\begin{eqnarray*}
\int_{\mathbb{R}}Mf_k(s)Mf_j(s)ds &\le & \Sigma^2H(2H-1)\frac{|k-j|^{2H-2}}{N^{2H}}\ \mbox{for} \ |k-j|>1\\
&\le & \Sigma^2 \frac{1}{2N^{2H}}c \ \mbox{for} \ |k-j|=1, \ c=(2^{2H}-1)\  \mbox{a constant}\\
&\le & \frac{\Sigma^2}{N^{2H}}\ \mbox{for} \ k=j
\end{eqnarray*}

Using product formula for Wiener chaos integrals (\ref{tp}) we get
\begin{eqnarray*}
\eta_k^2&=&I_1^2(f_k(s))\\
&=&I_2(f_k\otimes_0 f_k)+I_0(f_k\otimes_1 f_k)\\
&=&I_2(f_k\otimes_0 f_k)+\int_\mathbb{R}(M(\sigma_{s}\chi_{[\frac{k}{N},\frac{k+1}{N}]}))^2ds
\end{eqnarray*}
$$ \mbox{and} \ \eta_k^2-\theta_k=I_2(f_k\otimes_0 f_k)$$
We note that $I_0(f_k\otimes_1 f_k)=\langle f_k,f_k\rangle_{\mathcal{H}}=\int_\mathbb{R}(M(\sigma_{s}\chi_{[\frac{k}{N},\frac{k+1}{N}]}))^2ds$.

Then $\disp E(\eta_k^2)=EI_2(f_k\otimes_0 f_k)+\int_\mathbb{R}(M(\sigma_{s}\chi_{[\frac{k}{N},\frac{k+1}{N}]}))^2ds$. So, we get $\disp EI_2(f_k\otimes_0 f_k)=0$.

Let us now calculate the second moment of $\disp \sum^{N-1}_{k=0}\eta_k^2$.

\begin{eqnarray*}
\mathrm{E}\left(\sum^{N-1}_{k=0}\eta_k^2\right)^2&=&\sum^{N-1}_{k=0}\sum^{N-1}_{k'=0}\mathrm{E}\left((I_2(f_k\otimes_0 f_k)+\theta_k)(I_2(f_{k'}\otimes_0 f_{k'})+\theta_k)\right)\\
&=&\sum^{N-1}_{k=0}\sum^{N-1}_{k'=0}\left[\mathrm{E}(I_2(f_k\otimes_0 f_k)I_2(f_{k'}\otimes_0 f_{k'}))\right.\\
&&\qquad  +\theta_k \mathrm{E}I_2(f_{k'}\otimes_0 f_{k'})+ \theta_k' \mathrm{E}I_2(f_k\otimes_0 f_k)+\theta_k \theta_k']\\
&=&A_1+A_2+A_3+A_4
\end{eqnarray*}
Where $A_1, A_2, A_3, A_4$ are respective terms in the summation. Now $A_2=A_3=0$. We observe that 
Var$\disp \left(\sum^{N-1}_{k=0}\eta_k^2\right)=A_1$.

\begin{eqnarray*}A_1&=&2\sum^{N-1}_{k=0}\sum^{N-1}_{k'=0}\left< f_k\otimes_0 f_k,f_{k'}\otimes_0 f_{k'}\right>_{\mathcal{H}^{\otimes 2}}\\
&=&2\sum^{N-1}_{k=0}\sum^{N-1}_{k'=0} \left< f_k,f_{k'}\right>^2_{\mathcal{H}^{\otimes 1}}\\
&=&2\sum^{N-1}_{k=0}\sum^{N-1}_{k'=0} (\mathrm{E}(I_1(f_k)I_1(f_{k'}))^2\\
&=&2
\sum^{N-1}_{k=0}\sum^{N-1}_{k'=0}(\mathrm{E}(\eta_k\eta_{k'}))^2\\
&=& 2\sum^{N-1}_{k=0}\sum^{N-1}_{k'=0} (\int_{\mathbb{R}}Mf_k(s)Mf_{k'}(s)ds)^2
\end{eqnarray*}

\begin{eqnarray*}
A_1&\approx & 2(N^{1-4H}+\sum^{N-1}_{k=0}\mathop{\sum^{N-1}_{k'=0}}_{k\neq k'}\left(\Sigma^2 H(2H-1)\mid\frac{k'-k}{N}\mid^{2H-2}\frac{1}{N^2}\right)^2 )\\
&=&2(N^{1-4H}+\Sigma^4 H^2(2H-1)^2N^{-4H} \sum^{N-1}_{k=0}\mathop{\sum^{N-1}_{k'=0}}_{k\neq k'}\mid{k'-k}\mid^{4H-4})\\
&=&2(N^{1-4H}+\Sigma^4 H^2(2H-1)^2N^{-4H} \mathop{\sum^{N-1}_{k=-N+1}}_{k\neq 0}\sum^{N-1+k}_{k'=k}\mid k\mid^{4H-4})\\
&=&2N^{1-4H}+4\Sigma^4 H^2(2H-1)^2N^{-4H}(N-1) \sum^{N-1}_{k=1}k^{4H-4}\ \ \mbox{for} \ H>1/2 \end{eqnarray*}
Then $N^{4H-1}A_1<\infty$ if $H<\frac{3}{4}$. So, we can see $\disp N^{4H-1}E \left(\sum^{N-1}_{k=0}\eta_k^2\right)^2= N^{4H-1}A_1$. 

Let us write $X_N$ in terms of multiple Wiener Ito integral.
\begin{eqnarray*}
X_N &=&\sqrt{N}\left(N^{2H-1}\sum_{k=0}^{N-1}(I_1(f_k))^2-\tilde{\theta}_N \right)\\
&=&\sqrt{N}\left(N^{2H-1}\sum_{k=0}^{N-1}(I_2(f_k\otimes_0 f_k)+\theta_k )-\tilde{\theta}_N \right)\\
&=&N^{2H-\frac{1}{2}}I_2(\sum_{k=0}^{N-1}(f_k\otimes_0 f_k))+\sqrt{N}\Big(N^{2H-1}\sum_{k=0}^{N-1}\theta_k-\tilde{\theta}_N\Big)\\
&=& Y_N+\sqrt{N}(N^{2H-1}\sum_{k=0}^{N-1}\theta_k-\tilde{\theta}_N)
\end{eqnarray*}
Now $\disp\sqrt{N}(N^{2H-1}\sum_{k=0}^{N-1}\theta_k-\tilde{\theta}_N)= 0$ for all $N$.
Observe that $E(Y_N^2)=N^{4H-1}A_1=S_N$.
Define $\disp G_N=\frac{Y_N}{\sqrt{S_N}}$.

To prove asymptotic normality we will use the two theorems (\ref{md}) and (\ref{fm}). Using the theorems we want to show that $\|DG_N\|^2_{\mathcal{H}}\rightarrow 2$ in $L^2$. For that matter we first show
$\disp\lim_{N\rightarrow \infty}E[\|DG_N\|^2_{\mathcal{H}}]= 2$ and then $\disp\lim_{N\rightarrow \infty}E[\|DG_N\|^2_{\mathcal{H}}-2]^2=0$. Now,
\begin{eqnarray*}
&&[\|DG_N\|^2_{\mathcal{H}}-2]^2\\
&=&[\|DG_N\|^2_{\mathcal{H}}-E[\|DG_N\|^2_{\mathcal{H}}]+E[\|DG_N\|^2_{\mathcal{H}}]-2]^2\\
&=&[\|DG_N\|^2_{\mathcal{H}}-E[\|DG_N\|^2_{\mathcal{H}}]^2+[E[\|DG_N\|^2_{\mathcal{H}}]-2]^2+2[\|DG_N\|^2_{\mathcal{H}}-E[\|DG_N\|^2_{\mathcal{H}}][E[\|DG_N\|^2_{\mathcal{H}}]-2]\\
&=&A+B+2C
\end{eqnarray*}
where $A, B, C$ are respective terms. Now $EC=0, B\rightarrow 0$ as $ N\rightarrow\infty$. We will be interested in $A$ for further analysis.

Using (\ref{dd}) we get $$ \disp D_{t}Y_N=2 N^{2H-\frac{1}{2}} \sum_{j=0}^{N-1}f_j(t)I_1(f_j).$$ So $$ \|DY_N\|_{\mathcal{H}}^2=\disp 4 N^{4H-1}\sum_{k=0}^{N-1}\sum_{j=0}^{N-1} I_1(f_j)I_1(f_k)\langle f_j,f_k\rangle_{\mathcal{H}}.$$
Now $$E[\|DY_N\|_{\mathcal{H}}^2]=\disp 4N^{4H-1}\sum_{k=0}^{N-1}\sum_{j=0}^{N-1}(\langle f_j,f_k\rangle_{\mathcal{H}})^2=4N^{4H-1} \frac{A_1}{2}=2 S_N$$
We note that $E[\|DY_N\|_{\mathcal{H}}^2]=2E[Y_N^2]$.
So $E[\|DG_N\|_{\mathcal{H}}^2]\rightarrow 2$ as $N\rightarrow\infty$.
Let us calculate the following
\begin{eqnarray*}
&&\|DY_N\|_{\mathcal{H}}^2-E[\|DY_N\|_{\mathcal{H}}^2]\\
&=&4 N^{4H-1}\sum_{k=0}^{N-1}\sum_{j=0}^{N-1}[ (I_2(f_j\otimes_0 f_k)+I_0(f_k\otimes_1 f_k)\langle f_j,f_k\rangle_{\mathcal{H}}-\langle f_j,f_k\rangle_{\mathcal{H}}^2]\\
&=&4 N^{4H-1}\sum_{k=0}^{N-1}\sum_{j=0}^{N-1}[ I_2(f_j\otimes_0 f_k)\langle f_j,f_k\rangle_{\mathcal{H}}]
\end{eqnarray*}
And then \begin{eqnarray*}
E[\|DY_N\|_{\mathcal{H}}^2-E[\|DY_N\|_{\mathcal{H}}^2]]^2&=&16 N ^{8H-2} E[\sum_{k=0}^{N-1}\sum_{j=0}^{N-1}[ I_2(f_j\otimes_0 f_k)\langle f_j,f_k\rangle_{\mathcal{H}}]]^2\\
&\leq & \mbox{constant}\ \ N^{8H-6}\rightarrow \ 0
\end{eqnarray*}
For last part of the calculation see lemma 5.2 of Tudor (2013).
as $ N\rightarrow \infty $ and $H<\frac{3}{4}$. Hence the proof.
\end{proof}

\begin{theorem}\label{thm3} Under the conditions of theorem \ref{thm2}, \begin{equation}\sqrt{N}(\hat{\sigma^2}- \tilde{\theta}_N) \Longrightarrow \mathcal{N}(0,\sigma^2_{H,2,*})\end{equation} where $\sigma^2_{H,2,*}$ is a constant that can be computed explicitly, given the form of $\sigma(t)$.
\begin{proof}   Under the condition of $\mu=0$, the solution (\ref{eqn4}) of the stochastic differential equation (\ref{eqn1}) simplifies to \begin{equation}S_t = S_0 \exp\left(\int^t_0\sigma_sdB^H_s  -\frac{1}{2}  \int_\mathbb{R}(M(\sigma_s\chi_{[0,t]}))^2ds\right).\label{mo}
\end{equation}
Let us denote $ \tilde f_k(s)=\disp \sum_{j=0}^kf_k(s) $ and $\disp\delta_k=\int_{\mathbb{R}}(M\tilde f_{k+1}(s))^2-\int_{\mathbb{R}}(M\tilde f_{k}(s))^2$.
Putting the solution (\ref{mo}) in the definition of $\hat{\sigma^2}$ in equation (\ref{eqn5}), we get,
\begin{eqnarray}\label{eqn12}\hat{\sigma^2}& =&\frac{1}{Nh^{2H}}\sum^{N-1}_{k=0}\left[\eta_k-\frac{1}{2}\delta_k\right]^2\nonumber\\
&=&N^{2H-1}
\left[\sum^{N-1}_{k=0}\eta_k^2-\sum^{N-1}_{k=0}\eta_k \delta_k +\frac{1}{4}\sum^{N-1}_{k=0}\delta_k^2\right]
\end{eqnarray}
Combining (\ref{eqn12}) and (\ref{eqn13}), we have
\begin{eqnarray}\sqrt{N}(\hat{\sigma^2}-\tilde{\theta}_N)&=&X_N
-N^{2H-\frac{1}{2}}\sum^{N-1}_{k=0}\eta_k \delta_k+\frac{1}{4}N^{2H-\frac{1}{2}}\sum^{N-1}_{k=0}\delta_k^2.
\end{eqnarray}
where $X_N$ is defined in theorem \ref{thm2}. It is shown in lemma \ref{lem3} that the second term converges to zero in $L^2$ as $N\rightarrow\infty$. In lemma \ref{lem4} it is shown that the third term converges to zero. The theorem now follows by applying Chebyshev and Slutsky with theorem \ref{thm2} as before.
\end{proof}
\end{theorem}

\begin{lemma}\label{lem3} Under the assumptions of theorem \ref{thm2}, \begin{equation}G=N^{2H-\frac{1}{2}}\sum^{N-1}_{k=0}\eta_k \delta_k\stackrel{L^2}{\longrightarrow} 0
\end{equation}
\begin{proof}
Let us recall
\begin{equation*}\mathrm{E}\eta_k\eta_{k'}=\int_{\mathbb{R}}Mf_k(s)Mf_{k'}(s)ds
\end{equation*}
\begin{eqnarray*}
\delta_k&=&\int_{\mathbb{R}}[(M\tilde f_{k+1}(s))^2-(M\tilde f_{k}(s))^2]ds
\end{eqnarray*}
For $\disp H>\frac{1}{2}$ 
\begin{eqnarray*} 
\delta_k&=& H(2H-1)[\int_0^{\frac{k+1}{N}}\int_0^{\frac{k+1}{N}}\sigma(s)\sigma(t)|s-t|^{2H-2}ds dt-\int_0^{\frac{k}{N}}\int_0^{\frac{k}{N}}\sigma(s)\sigma(t)|s-t|^{2H-2}ds dt]\\
&=& H(2H-1)[\int_{\frac{k}{N}}^{\frac{k+1}{N}}\int_{\frac{k}{N}}^{\frac{k+1}{N}}+2\int_0^{\frac{k}{N}}\int_{\frac{k}{N}}^{\frac{k+1}{N}}]
\end{eqnarray*}
So 
\begin{eqnarray*} 
|\delta_k| &\leq & \frac{\Sigma^2}{N^{2H}}[1+(k+1)^{2H}-(k)^{2H}-1]\\
&\approx & \frac{\Sigma^2}{N^{2H}} (k+1)^{2H-1}
\end{eqnarray*}
We look at the $L_2$ norm of $G$.
\begin{eqnarray*}E(G^2)&=&N^{4H-1} \sum^{N-1}_{k=0}\sum^{N-1}_{k'=0}\delta_k\delta_{k'}\int_{\mathbb{R}}Mf_k(s)Mf_{k'}(s)ds\\
&\rightarrow& 0.
 \end{eqnarray*}
using estimates for $\delta_k$ and  estimates of $\disp \int_{\mathbb{R}}Mf_k(s)Mf_{k'}(s)ds$.
\end{proof}
\end{lemma}

\begin{lemma}\label{lem4} Under the assumptions of theorem \ref{thm2}, then  \begin{equation}N^{2H-\frac{1}{2}}\sum^{N-1}_{k=0}\delta_k^2\rightarrow 0
\end{equation}

\begin{proof}
Again using the estimates of $\delta_k$ we have 
\begin{eqnarray*}
N^{2H-\frac{1}{2}}\sum^{N-1}_{k=0}\delta_k^2&\le&N^{2H-\frac{1}{2}}\sum^{N-1}_{k=0}\left(\frac{\Sigma^2}{N^{2H}} (k+1)^{2H-1}\right)^2\\
&=&\Sigma^4 N^{-2H-\frac{1}{2}}N^{4H-1}
\end{eqnarray*}
This proves the lemma for $\disp H>\frac{1}{2}$ as well as $\disp H<\frac{3}{4}$.
\end{proof}\end{lemma}
\subsection{Simulation studies (time varying volatility)}\label{sec:sim2}
In this section we did simulation studies to see the difference between our actual parameter of interest $\int_{\mathbb{R}}(M(\sigma(s)\chi_{[0,1]}))^2ds$ and what we achieve $\tilde{\theta}_N$ for sample size $N$, for different Hurst parameter with different $\sigma(s)$ function. We have chosen different functions $\sigma(t)$, necessarily bounded, on the interval $[0,1]$, calculate $\tilde{\theta}_N$ and report the results. We take $N=1000$ and compute $\tilde{\theta}_N$ and $\theta$ and note the difference.

Let us consider $\sigma(t)$ a sub linear function of the form $\sigma(t)=\sigma t^{\alpha},\ t\in (0,1)$, $\frac{1}{2}<H<\frac{3}{4}$, $\sigma>0,\ 0<\alpha<1$.
\begin{table}
\caption{The comparison of $\tilde{\theta}_N$ and $\theta$}
\begin{center}
\begin{tabular}{|c|c|c|c|} \hline
  \multicolumn{4}{|c|}{$\sigma$=0.4, $\alpha=0.3$} \\  \hline
     H & 0.55   & 0.65   & 0.74  \\
   $\tilde{\theta}_N$ & 0.1000039  &  0.1000005  & 0.09992656  \\
   $\theta$ & 0.09962605 & 0.09868432  &   0.09770835 \\  \hline
   \multicolumn{4}{|c|}{$\sigma$=6.4, $\alpha=0.3$} \\  \hline
     H & 0.55   & 0.65   & 0.74  \\
   $\tilde{\theta}_N$ &  25.60027  & 25.60061   & 25.6002  \\
   $\theta$ & 25.50445	& 25.2632  &   25.01328 	 \\  \hline
\end{tabular}
\end{center}
\end{table}
Next we consider functions of the form $\sigma(t)=\sigma(t^{\alpha}+t^{\beta})$ for $t\in(0,1)$, $0<\alpha<1$ and $\beta>1$, i.e. polynomial with positive fraction and integer powers. We note that $\tilde{\theta}_N$ may not converge as $N\rightarrow\infty$ and it will indeed not converges. So the simulation result is only for $N=1000$.
\begin{table}
\caption{The comparison of $\tilde{\theta}_N$ and $\theta$}
\begin{center}
\begin{tabular}{|c|c|c|c|} \hline
  \multicolumn{4}{|c|}{$\sigma$=0.4, $\alpha=0.8,\beta=2$} \\  \hline
     H & 0.55   & 0.65   & 0.74  \\
   $\tilde{\theta}_N$ & 0.1777476  & 0.1777324  & 0.1776835  \\
   $\theta$ & 0.1705813  & 0.1579971   &  0.1482963  \\  \hline
   \multicolumn{4}{|c|}{$\sigma$=6.4, $\alpha=0.8,\beta=2$} \\  \hline
     H & 0.55   & 0.65   & 0.74  \\
   $\tilde{\theta}_N$ & 45.50396  & 45.50438  &  45.5041 \\
   $\theta$ & 43.66903	& 40.44747  &  37.96394 \\  \hline
\end{tabular}
\end{center}
\end{table}

For practical purpose the sub linear functions seems best.
\section{\bf Conclusions}\label{sec6}
In this paper we sketch the way to obtain the option price for fBm driven model with time varying volatility. We identify the parameter of interest for calculating option price. Next we have proposed estimator from high frequency data for parameter similar to so called ''integrated volatility", in case of constant volatility and time varying volatility model driven by fBm. We have shown that estimators are asymptotically normally distributed for $H<\frac{3}{4}$. For time varying volatility model, the estimator will not asymptotically unbiased for our parameter of interest. Through some simulation study we showed how close of the parameter of interest can be achieved by the estimators under consideration. 
\subsection{Future directions}
\begin{enumerate}
\item In all these we assume $H$ as a known quantity. The estimation for $H$ also exists separately. See
Prakasa Rao (2010), Breton et al. (2009). Is there any way to
combine?

\item Following Zhang et al (2005) we want to develop method of inferring volatility when the process is observed discretely with noise.

\item Following Barndorff-Nielsen and Shephard (2004) we want to consider estimators for jump process.
\end{enumerate}
\subsection{Comments}
\begin{itemize}
\item Why consider fBM driven models? Non-stationary time series will also take care of thick-tails and long-range dependence in returns. But it is not easy to put them in an option pricing framework.
    \item Why do we have confidence intervals for option prices, when looked at as solution of an optimization problem? There is uncertainty in utility/preferences.
\end{itemize} 
\section{Acknowledgement}
First author wants to acknowledge Department of Science and Technology, India, for financial support to conduct this research work.
\section{Appendix}
\subsection{Wiener integral \label{WI}}

$\Omega:=\mathcal{S}'(\mathbb{R})$, dual of the space of Schwartz class functions $\mathcal{S}$, is tempered distributions with sigma algebra $\mathcal{F}$. $ \langle\omega,f\rangle$ is the random variable by action of $\omega \in \Omega$ on $f\in \mathcal{S}(\mathbb{R})$. Bochnor Minlos probability measure $P$ on $(\Omega,\mathcal{F})$ is such that $\disp E[\exp\  i \langle\omega,f\rangle]=e^{-\frac{1}{2}\|f\|^2},\ \|f\|^2=\int_{\mathbb{R}}f^2(x)dx$. We also have expectation $E[\langle\omega,f\rangle]=0$ and variance $ E[\langle\omega,f\rangle]^2=\|f\|^2$ under $P$. Let $I(0,t)$ is indicator function. Then $\tilde B_t(\omega)= \langle\omega,I(0,t)\rangle \in L^2(P)$ and $\tilde B_t$ is Gaussian random variable for each $t$.  Using Kolmogorov's continuity theorem $\tilde B_t$ has continuous version as $B_t$ and $B_t$ is standard Brownian motion. This duality can be extended for $f\in L^2(\mathbb{R})$ approximating $f$ by step functions we get $\disp \langle\omega,f\rangle=\int_{\mathbb{R}}f(t) dB_t(\omega)$. 

Let $M(0,t)=M I(0,t)\in L^2(\mathbb{R})$. Then $\tilde B_t^H (\omega)=\langle\omega,M(0,t)\rangle$. $\tilde B_t^H$ is Gaussian random variable with $E\tilde B_t^H=0$ and $\disp E\tilde B_t^H\tilde B_s^H=\int_{\mathbb{R}}M(0,t)(x)M(0,s)(x)dx=\frac{1}{2}(t^{2H}+s^{2H}-|t-s|^{2H})$. Again take continuous version of $\tilde B_t^H$ as $B_t^H$. So we get $\disp B^H_t=\int_{\mathbb{R}}(MI(0,t)(s)dB_s$.

Let $f\in L^2_H(\mathbb{R})$, deterministic. Then $Mf \in L^2(\mathbb{R})$. 
The Wiener integral with respect to fractional Brownian motion is defined as
\begin{equation}
\int_{\mathbb{R}}f(s)dB_s^H =\int_{\mathbb{R}}(Mf)(s)dB_s
\end{equation}

Take $f\in \mathcal{S}$. Construct Bochnor Minlos probability measure $P^H$ on $(\Omega,\mathcal{F})$ is such that $\disp E[\exp\  i \langle\omega,f\rangle_H]=e^{-\frac{1}{2}\|f\|_H^2},\ \|f\|_H^2= \|f\|_{L^2_H}=\|Mf\|_{L^2}=\int_{\mathbb{R}}(Mf)^2(x)dx$. We also have expectation $E^H[\langle\omega,f\rangle_H]=0$ and variance $ E^H[\langle\omega,f\rangle_H]^2=\|f\|_H^2$ under $P^H$. Let $I(0,t)$ is indicator function. Then $\tilde B^H_t(\omega)= \langle\omega,I(0,t)\rangle_H \in L^2(P^H)$ and $\tilde B^H_t$ is Gaussian random variable for each $t$.  Using Kolmogorov's continuity theorem $\tilde B^H_t$ has continuous version as $B^H_t$ and $B^H_t$ is fractional Brownian motion. This duality can be extended for $f\in L^2_H(\mathbb{R})$ 
we get $\disp \langle\omega,f\rangle_H=\int_{\mathbb{R}}f(t) dB^H_t(\omega)$. 
So, we note that the term $\disp \int^t_0\sigma_sdB^H_s$ appears in solution of the SDE (\ref{eqn4}) is a Wiener integral with respect to Brownian motion. The first two moments of Wiener integral with respect to fractional Brownian motion are as follows 
$$E(\int_{\mathbb{R}}f(s)dB_s^H)=0$$ and $$ E[\int_{\mathbb{R}}f(s)dB_s^H]^2=\|f\|_{L^2_H}=\|Mf\|_{L^2}$$ $$E[\int_{\mathbb{R}}f(s)dB_s^H \int_{\mathbb{R}}g(s)dB_s^H]=\langle f,g\rangle_{L^2_H}$$

\subsection{Wick product and related topics \label{WP}}
This section consists of background material required for section \ref{sec4}.

{\bf Chaos expansion Theorem} Let $F\in L^2(P)$. Then there exists a unique family ${c_{\alpha}}, \alpha \in \mathcal{I}$ of constants, $c_{\alpha} \in \mathbb{R}$ such that $F(\omega) = \disp\sum_{\alpha \in \mathcal{I}} c_{\alpha} H_{\alpha}(\omega)$, $H_{\alpha}$ is multi indexed Hermite polynomial of Bm (convergence in $L^2(P)$). Moreover, we have the isometry $E(F^2)=\disp\sum_{\alpha \in \mathcal{I}} c_{\alpha}^2 \alpha !  $. 

Let us define $(\mathcal{S})$ as {\bf space of stochastic test functions} and $(\mathcal{S}^*)$ as {\bf space of stochastic distributions}. 

$(\mathcal{S})$ is collection of all $F\in L^2(P)$ such that it's expansion is $F(\omega)=\disp\sum_{\alpha \in \mathcal{I}}c_{\alpha} H_{\alpha}(\omega)$ where $\|F\|_k^2=\disp \sum_{\alpha}\alpha!c_{\alpha}^2(2\mathbb{N})^{k \alpha}<\infty$ for all integer $k=1,2,\cdots$ with $\disp (2\mathbb{N})^{k \gamma}=\prod_{j=1}^m (2j)^{\gamma_j}, \ \gamma=(\gamma_1,\cdots, \gamma_m)\in \mathcal{I}$.

$(\mathcal{S}^*)$ is collection of all $G\in L^2(P)$ such that it's expansion is $G(\omega)=\disp\sum_{\beta \in \mathcal{I}}c_{\beta} H_{\beta}(\omega)$ where $\|G\|_q^2=\disp \sum_{\beta}\beta!c_{\beta}^2(2\mathbb{N})^{-q  \beta}<\infty$ for some integer $q<\infty$. 

$(\mathcal{S}^*)$ is dual of $(\mathcal{S})$ with duality relation as follows: 

If $F(\omega) = \disp\sum_{\alpha \in \mathcal{I}} a_{\alpha} H_{\alpha}(\omega)\in (\mathcal{S})$ and $G(\omega) = \disp\sum_{\alpha \in \mathcal{I}} b_{\alpha} H_{\alpha}(\omega) \in (\mathcal{S}^*)$ then action of $G$ on $F$ is $\langle G,F\rangle_{(\mathcal{S}^*),(\mathcal{S})}=\disp \sum_{\alpha \in \mathcal{I}}\alpha !a_{\alpha} b_{\alpha}$.

If $L^2(P)\subset (\mathcal{S}^*) $ and $(\mathcal{S}) \subset L^2(P)$ then action of $G$ on $F$ is $\langle G,F\rangle_{(\mathcal{S}^*),(\mathcal{S})}=\langle G,F\rangle_{L^2(P)}=E(GF)$.

If $F(\omega) = \disp\sum_{\alpha \in \mathcal{I}} a_{\alpha} H_{\alpha}(\omega)\in (\mathcal{S}^*)$ and $G(\omega) = \disp\sum_{\beta \in \mathcal{I}} b_{\beta} H_{\beta}(\omega) \in (\mathcal{S}^*)$ then {\bf Wick product} $\diamond$ is defined as $(F\diamond G)(\omega)=\disp\sum_{\alpha,\beta \in \mathcal{I}} a_{\alpha}b_{\beta} H_{\alpha+\beta}(\omega)\in (\mathcal{S}^*)$.

{\bf Fractional white noise}  $\disp W^H_t=\frac{dB^H_t}{dt} $ is an element of $(\mathcal{S}^*) $ see Elliott and Van der Hoek (2003), Biagini {\it et al.} (2004) for detail.

{\bf Wick Ito Skorohod integral with respect to $B^H_t$: }
If $Y:\mathbb{R}\rightarrow (\mathcal{S}^*)$ is such that $Y_t\diamond W^H_t$ is integrable in $(\mathcal{S}^*)$ then we define 
\begin{equation}
\int_{\mathbb{R}}Y_t dB^H_t=\int_{\mathbb{R}}Y_t\diamond W^H_t dt.
\end{equation}

If $f\in L^2_H(\mathbb{R})$ then $$\int_{\mathbb{R}}f(s)dB^H_s= \int_{\mathbb{R}}f(s)\diamond W^H_s ds=\int_{\mathbb{R}}(Mf)(s)dB_s. $$
\subsection{Background to deal with time varying volatility estimator \label{tv}}
In this section we introduce some notations and established results which will be needed for our future calculation.
Our fractional Brownian motion $B_t^H$ is centered, continuous, mean zero Gaussian processes with covariance functions as  $R_B^H=cov(B_s^H, B_t^H)=\frac{1}{2}[t^{2H}+s^{2H}-|t-s|^{2H}]$.  Let $\mathcal{E}$ be the set of real valued step functions.
For $\phi=I[0,t], \psi=I[0,s] \in \mathcal{E}$ let us define inner product $\langle\phi,\psi\rangle_{\mathcal{E}}=\langle I_{[0,s]},I_{[0,t]}\rangle_{\mathcal{E}}=R_B^H$. For $\phi=\sum_j a_jI[0,t_j]$, set $ B^H(\phi)=\sum_j a_jB^H_{t_j}$. Let $\psi=\sum_j b_jI[0,t_j]$. So, $E(B^H(\phi)B^H(\psi))=\langle\phi,\psi\rangle_{\mathcal{E}}$. Next for $\phi\in \mathcal{H}$, there are $\phi_n \in \mathcal{E}$ such that $\phi_n\rightarrow \phi$ in $\mathcal{H}$ then $B^H(\phi) $ is the $L^2$ limit of  $B^H(\phi_n) $. So we get $\langle\phi,\psi\rangle_{\mathcal{H}}=EB^H(\phi)B^H(\psi)$. $\{B^H(\phi), \phi \in \mathcal{H}\}$ is called isonormal Gaussian process.

Let $H_n$ be $n$ th Hermite polynomial satisfying \begin{equation}
 \disp \frac{d}{dx}H_n(x)=H_{n-1}(x),\ n\geq 1.
\end{equation} Take $\phi \in \mathcal{H}$ such that $\|\phi\|_{\mathcal{H}}=1$. Consider random variables $H_n(B^H(\phi))$ and take the closure of the span of these random variables. This is the $n$ th order Wiener chaos $\mathcal{W}_n$.

$I_n$, the multiple stochastic (Wiener Ito) integral with respect to isonormal Gaussian process $B^H$, is a map from $\mathcal{H}^{\odot n}$ to $\mathcal{W}_n$, $\mathcal{H}^{\odot n}$ being symmetric tensor product of $\mathcal{H}$.  $ \mathcal{H}^{\odot n}$ has norm $ \disp \frac{1}{\sqrt{n !}}\|.\|_{\mathcal{H}^{\otimes n}}$, $\mathcal{H}^{\otimes n}$ is tensor product of $\mathcal{H}$.

Then for $f\in \mathcal{H}^{\odot n}$,  we also have $I_n(f)=I_n(\tilde f)$,  $\tilde{f}$ is symmetrization of $f$.

For $\phi \in \mathcal{H},\ \disp I_n(\phi^{\otimes n})=n! H_n(I_1(\phi))= n! H_n(B^H(\phi))$ is linear isometry between $\mathcal{H}^{\odot n} $ and $\mathcal{W}_n$.

Now for $f\in \mathcal{H}^{\odot n}$ and $g\in \mathcal{H}^{\odot m}$ we have followings:
\begin{eqnarray}
 E(I_n(f)I_m(g))&=&n!\langle\tilde{f},\tilde{g}\rangle _{\mathcal{H}^{\otimes n}}\ \mbox{ if } \ m=n\\
 E(I_n(f)I_m(g))&=&0\ \mbox{ if }\ m \neq n
\end{eqnarray}

Let $\{e_i, i\geq 1\}$ be an orthonormal basis of $\mathcal{H}$, $m,n \geq 1, \ r=0,\cdots, n\wedge m $. $f\otimes_r g \in \mathcal{H}^{\otimes (m+n-2r)}$ is contraction is defined as \begin{equation}
f\otimes_r g=\sum_{i1,\cdots,ir=1}^{\infty}\langle f, e_{i1}\otimes\cdots\otimes e_{ir}\rangle_{\mathcal{H}^{\otimes r}}\langle g, e_{i1}\otimes\cdots\otimes e_{ir}\rangle_{\mathcal{H}^{\otimes r}}. \label{cc}
\end{equation} This definition does not depend on the choice of orthonormal basis and $\langle f, e_{i1}\otimes\cdots\otimes e_{ir}\rangle_{\mathcal{H}^{\otimes r}}\in \mathcal{H}^{\odot (n-r)}$, $\langle g, e_{i1}\otimes\cdots\otimes e_{ir}\rangle_{\mathcal{H}^{\otimes r}}\in \mathcal{H}^{\odot (m-r)}$. $f\otimes_r g $ is not necessarily symmetric. Let $f\tilde \otimes_r g$ is symmetrization of $f\otimes_r g$. Then 
\begin{equation}\disp I_n(f)I_m(g)= \sum_{r=0}^{m \wedge n} r!  \Bigl(\begin{matrix} n\\ r \end{matrix} \Bigr) \Bigl(\begin{matrix} m\\ r \end{matrix} \Bigr) I_{n+m-2r}(f\tilde\otimes_r g ).\label{tp}
\end{equation}
Also for $n=m=r$ we have \begin{equation}
I_0(f\otimes_r g)=\langle f,g\rangle_{\mathcal{H}^{\otimes r}}.
\end{equation}

Let $F$ be a functional of the isonormal Gaussian process $B^H$ such that $E(F(B^H)^2)<\infty$ then there is unique sequence $f_n\in \mathcal{H}^{\odot n}$ and $F$ can be written as sum of multiple stochastic integrals as $ \disp F=\sum_{n\ge 0}I_n(f_n)$ with  and $I_0(f_0)=E(F)$ where the series converges in $L^2$

For $\phi_1, \cdots, \phi_n \in \mathcal{H}$, let $F=g(B^H(\phi_1),\cdots,B^H(\phi_n))$ with $g$ smooth compactly supported. Then Malliavin derivative $D$ is $\mathcal{H}$ valued random variable
defined as follows:
\begin{equation}
DF=\sum_{i=1}^n\frac{\partial g}{\partial x_i}(B^H(\phi_1),\cdots,B^H(\phi_n))\phi_i.
\end{equation}
If $\mathcal{H}$ is $ L^2(\mathbb{R})$ for some non atomic measure then $ DF$ can be identified as follows: $DF=(D_tF)_{t\in \mathbb{R}}$
\begin{equation}
D_tF=\sum_{i=1}^n\frac{\partial g}{\partial x_i}(B^H(\phi_1),\cdots,B^H(\phi_n))\phi_i(t), \ t \in \mathbb{R}
\end{equation}
If $F=I_n(f),f \in \mathcal{H}^{\odot n}$, 
for every $t \in \mathbb{R}$, then \begin{equation}D_{t}F= D_{t}I_n(f)=n I_{n-1} f(.,t). \label{dd}
\end{equation} 
$I_{n-1}(f(.,t))$ means $n-1$ multiple stochastic integral is taken with respect to first $n-1$ variables $t_1,\cdots, t_{n-1}$ of $f(t_1,\cdots, t_{n-1},t), \ t$ is kept fixed.
For Malliavin calculus details, see Nualart (1995), Nourdin (2012).
To prove asymptotic normality we will use the following two theorems [5.1] and [5.2] taken from Tudor C.A. (2008):

\begin{theorem}
Let $I_n(f) $ be a multiple integral of order $n\ge 1 $ with respect to an isonormal process $M$. Then
$$ d(\mathcal{L}(I_n(f)),\mathcal{N}(0,1))\le c_n [E(|DI_n(f)|_{\mathcal{H}}^2-n)^2]^{\frac{1}{2}}$$ where $D$ is the Malliavin derivative with respect to $B^H$ and $\mathcal{H}$ is the canonical Hilbert space associated to $B^H$. Here $d$ can be any of the distances like Kolmogorov Smirnov distance, or total variation distance etc. and depending upon $d$ and the order $n$ one will end up a constant $c_n$.
$\mathcal{L}(B^H)$ stands for law of $B^H$. \label{md}
\end{theorem}
\begin{theorem}
Fix $n\ge 2$ and let $(F_k,k >1), \ F_k=I_n(f_k)$ (with $f_k \in \mathcal{H}^{\odot n}$, for every $k\ge 1$) be a sequence of square integrable random variables in the $n$th Wiener chaos of an isonormal process $B^H$ such that $E[F_k^2]^2\rightarrow 1$ as $k\rightarrow \infty$. Then the following are equivalent:
(i) The sequence $(F_k)_{k\ge 0}$ converges in distribution to the normal law $\mathcal{N}(0,1)$.
(ii) One has $E[F_k^4]\rightarrow 3$ as $k\rightarrow \infty$.
(iii) For all $1\le l \le n-1$ it holds that $\disp\lim_{k\rightarrow \infty}|f_k\otimes_l f_k|_{\mathcal{H}^{\otimes 2 (n-l)}}=0$.
(iv) $ |DF_k|^2_{\mathcal{H}}\rightarrow n \ \mbox{in}\ L^2 \ \mbox{as}\ k\rightarrow \infty$, where $D$ is the Malliavin derivative with respect to $B^H$.\label{fm}
\end{theorem}
The above theorems we will use to prove asymptotic normality for our proposed estimator.

\end{document}